\documentclass[a4paper,11pt]{amsart}
\usepackage{amsmath,amssymb,mathrsfs}
\date{2022/02/20}
\usepackage[alphabetic]{amsrefs}
\usepackage[arrow,matrix]{xy}
\usepackage{enumerate,booktabs}

\theoremstyle{plain}
\newtheorem{thm}{Theorem}[section]
\newtheorem{cor}[thm]{Corollary}
\newtheorem{lem}[thm]{Lemma}
\newtheorem{prop}[thm]{Proposition}
\newtheorem{conj}[thm]{Conjecture}

\theoremstyle{definition}
\newtheorem{defn}[thm]{Definition}
\newtheorem{example}[thm]{Example}
\newtheorem{rem}[thm]{Remark}

\newcommand{\cA}{\mathcal A}
\newcommand{\cC}{\mathcal C}
\newcommand{\cE}{\mathcal E}
\newcommand{\cO}{\mathcal O}
\newcommand{\cX}{\mathcal X}
\newcommand{\cM}{\mathcal M}
\newcommand{\bP}{\mathbb P}
\newcommand{\bZ}{\mathbb Z}
\newcommand{\bQ}{\mathbb Q}
\newcommand{\bC}{\mathbb C}
\newcommand{\fS}{\mathfrak S}
\DeclareMathOperator{\End}{End}
\DeclareMathOperator{\Gal}{Gal}
\DeclareMathOperator{\CH}{CH}
\DeclareMathOperator{\HS}{HS}
\DeclareMathOperator{\Aut}{Aut}
\DeclareMathOperator{\gr}{gr}
\DeclareMathOperator{\charac}{char}
\DeclareMathOperator{\GL}{GL}
\DeclareMathOperator{\Ortho}{O}
\DeclareMathOperator{\Frac}{Frac}
\DeclareMathOperator{\trace}{tr}
\DeclareMathOperator{\height}{ht}
\DeclareMathOperator{\Fix}{Fix}
\DeclareMathOperator{\Spec}{Spec}
\DeclareMathOperator{\sign}{sgn}
\DeclareMathOperator{\Km}{Km}
\DeclareMathOperator{\ord}{ord}
\newcommand{\id}{\mathrm{id}}
\newcommand{\crys}{\mathrm{crys}}
\newcommand{\Hcrys}{H_\crys}
\newcommand{\Het}{H_{\mathrm{\acute et}}}
\DeclareMathOperator{\Ker}{Ker}
\DeclareMathOperator{\Image}{Im}
\DeclareMathOperator{\NS}{NS}
\newcommand{\divides}{\mid}
\newcommand{\notdivides}{\nmid} %\nmid?
\newcommand{\pairing}[2]{\langle #1,#2 \rangle}
\newcommand{\set}[1]{\{#1\}}
\newcommand{\card}[1]{\lvert #1 \rvert}
\newcommand{\rationalto}{\dashrightarrow}
\newcommand{\injto}{\hookrightarrow}
\newcommand{\surjto}{\twoheadrightarrow}
\newcommand{\isomto}{\stackrel{\sim}{\to}}
\newcommand{\isom}{\cong}
\newcommand{\restrictedto}[1]{\rvert_{#1}}

\newcommand{\rI}{\mathrm I}
\newcommand{\rII}{\mathrm {II}}
\newcommand{\LK}[1]{L_{\mathrm{K#1}}}

\title
[Degeneration of K3 surfaces with automorphisms]
{Degeneration of K3 surfaces with non-symplectic automorphisms}

\author{Yuya Matsumoto}
\address{Department of Mathematics, Faculty of Science and Technology, Tokyo University of Science, 2641 Yamazaki, Noda, Chiba, 278-8510, Japan}
\email{\url{matsumoto.yuya.m@gmail.com}}
\email{\url{matsumoto_yuya@ma.noda.tus.ac.jp}}

\thanks{This work was supported by JSPS KAKENHI Grant Numbers 15H05738 and 16K17560.}

\begin{document}

\begin{abstract}
We prove that a K3 surface with an automorphism acting on the global $2$-forms by 
a primitive $m$-th root of unity, $m \neq 1,2,3,4,6$,
does not degenerate 
(assuming the existence of the so-called Kulikov models).
A key result used to prove this is
the rationality of the actions of automorphisms
on the graded quotients of the weight filtration of the $l$-adic cohomology groups of the surface.
\end{abstract}

\maketitle

\section{Introduction}

Let $\cO_K$ be a Henselian discrete valuation ring (DVR)
with fraction field $K$ and residue field $k$.
We consider the problem of degeneration of K3 surfaces:
Given a K3 surface $X$ over $K$,
we investigate its possible extensions $\cX$ over $\cO_K$
and their reductions $X_0$ over $k$.

The so-called Kulikov models, which are semistable models of $X$ over $\cO_K$ with nice properties,
is a standard tool to discuss degeneration of K3 surfaces.
The special fibers of Kulikov models are classified into three types: 
Type I, smooth K3 surfaces, and Types II and III, which are reducible surfaces satisfying certain conditions.
It is conjectured that any K3 surface over $K$ admits a Kulikov model after replacing $K$ by a finite extension,
but not yet proved in general.
See Section \ref{sec:Kulikov} for details.

In this paper we relate the properties of Kulikov models with (non-symplectic) automorphisms of $X$:

\begin{thm} \label{thm:main}
Assume $\charac K \neq 2$.
Let $\cX$ be a Kulikov model over $\cO_K$ of a K3 surface $X$ over $K$
and $X_0$ its special fiber over $k$. 
Denote by $m = m(X)$ the order of the image of $\rho \colon \Aut_K(X) \to \GL (H^0(X, \Omega^2_{X/K}))$ (which is finite).

\textup{(1)} Assume $m \neq 1,2,3,4,6$. Then $X_0$ is of Type I, i.e.\ a smooth K3 surface.

\textup{(2)} Assume $m \neq 1,2$. Then $X_0$ is either of Type I or II.
\end{thm}

The key idea of the proof is describing the action of $\Aut(X)$ on the $n$-th graded quotients $\gr^W_n$ of the weight filtration of $\Het^2(X_{\overline K}, \bQ_l)$
in terms of $\rho$ (Lemma \ref{lem:action on gr}).
We also show the rationality (and $l$-independence) of such action (Theorem \ref{thm:action on gr}),
and using this we derive some restrictions on $\dim \gr^W_n$. 
By calculating $\dim \gr^W_n$ (using the classification of the special fibers of Kulikov models, see Section \ref{sec:Kulikov})
we can exclude certain types of degeneration,
and then the remaining possibilities are as stated in the theorem.

The assumption on $m$ in Theorem \ref{thm:main} is optimal: see Examples \ref{ex:Type III}--\ref{ex:Type II}.

\medskip

In Section \ref{sec:application} we give an application on K3 surfaces with non-symplectic automorphisms of prime order $p \geq 5$:
we can show that the moduli space of such K3 surfaces (in characteristic $0$) is compact
and that any such surface defined over a number field has everywhere potential good reduction.
Since these moduli spaces for $p = 5,7,11$ have positive dimension $4,2,1$ respectively,
there are plenty of such surfaces. 

\medskip

We also have a conjectural generalization of Theorem \ref{thm:main}:
\begin{conj} \label{conj:E-version}
Assume $\charac K = 0$.
Let $X$ and $\cX$ be as in the previous theorem.
Let $E$ be the Hodge endomorphism field of $X$.

\textup{(1)} Assume $E$ is not $\bQ$ nor an imaginary quadratic field. Then $X_0$ is of Type I.

\textup{(2)} Assume $E \neq \bQ$. Then $X_0$ is either of Type I or II.
\end{conj}
\begin{thm} \label{thm:E-version}
If the Hodge conjecture for $X \times X$ is true,
then Conjecture \ref{conj:E-version} for $X$ is true.
\end{thm}
It is known that the Hodge conjecture for $X \times X$ is true if $E$ is a CM field.
See Section \ref{sec:proof} for the definition of the Hodge endomorphism field. It is either a totally real field or a CM field.
\begin{rem}
The Hodge endomorphism field of $X$ clearly contains $\bQ(\zeta_{m})$, where $m = m(X)$.
The cyclotomic field $\bQ(\zeta_m)$ is a CM field for $m \neq 1,2$
and is imaginary quadratic only if $m = 3,4,6$. 
Hence Theorem \ref{thm:E-version} generalizes Theorem \ref{thm:main}.
\end{rem}

\medskip

\subsection*{Acknowledgments}
I am grateful to Yuji Odaka
for the discussion from which this study originated (see Remark \ref{rem:orientation}).
I thank Simon Brandhorst, Kazuhiro Ito, Tetsushi Ito, Teruhisa Koshikawa, Yukiyoshi Nakkajima, and Takeshi Saito for helpful comments and discussions.
I thank the anonymous referee for helpful suggestions.

\section{Transcendental lattices and $2$-forms of K3 surfaces}

In this section $X$ is a K3 surface over a field $k$
and $l$ is an arbitrary prime different from $\charac k$.

Let $\rho \colon \Aut_k(X) \to \GL (H^0(X, \Omega^2_{X/k})) = k^*$ be the natural action
(we have $\dim H^0(X, \Omega^2_{X/k}) = 1$ since $X$ is a K3 surface).
An element of $\Aut_k(X)$ is called \emph{symplectic} if it belongs to $\Ker \rho$.
\begin{lem} \label{lem:rho}
$\Image \rho$ is a finite (cyclic) group.
\end{lem}
We denote by $m(X)$ (resp.\ $m(g)$) the order of the group $\Image \rho$ (resp.\ of the element $\rho(g)$).
We denote by $\mu_m$ the group of $m$-th roots of $1$.
\begin{proof}
Characteristic $0$: 
Finiteness follows from a general result of Ueno \cite{Ueno:classification}*{Theorem 14.10}.
Nikulin \cite{Nikulin:factorgroups}*{Theorem 10.1.2} also showed finiteness and moreover showed $\phi(m(X)) \leq 20$,
where $\phi(m) = \# (\bZ/m\bZ)^*$ 
is the number of invertible classes modulo $m$.
(In particular we have $m(X) \leq 66$.)

Characteristic $p > 0$, supersingular: Nygaard \cite{Nygaard:higherdeRham-Witt}*{Theorem 2.1} (see Remark \ref{rem:Nygaard p=2}) showed that
$\rho(g) \in \mu_{p^{\sigma_0} + 1}$ for every $g$, where $\sigma_0$ is the Artin invariant of $X$ (which is a positive integer $\leq 10$). 
Hence $m(g)$ and $m(X)$ divide $p^{\sigma_0} + 1$. 

Characteristic $p > 0$, finite height: 
Let $W = W(\overline k)$ be the ring of Witt vectors over $\overline k$
and let $K = \Frac W$.
Lieblich--Maulik \cite{Lieblich--Maulik:cone}*{Corollary 4.2} showed that
there exists a lifting $\tilde X$ over $W$
such that the specialization morphism $\NS(\tilde X_{\overline K}) \to \NS(\tilde X_{\overline k})$ is an isomorphism.
They also showed \cite{Lieblich--Maulik:cone}*{Theorem 2.1} that for such $\tilde X$ the restriction map 
$\Aut(\tilde X) \to \Aut(X_K)$ is an isomorphism,
and the same assertion holds for $\tilde X_R := \tilde X \otimes R$ for any finite extension $R$ of $W$.
They also showed \cite{Lieblich--Maulik:cone}*{Section 6} that for such $\tilde X$ the specialization map 
$\Aut(\tilde X_{\overline K}) \to \Aut(X_{\overline k})$,
defined as the limit of $\Aut(\tilde X_{\Frac R}) \stackrel{\sim}{\leftarrow} \Aut(\tilde X_R) \to \Aut(X_{\overline k})$,
has finite cokernel.
Comparing the actions on a 2-form on $\tilde X_R$ and its mod $p$ reduction,
we observe that this specialization map is compatible with $\rho$.
The assertion is reduced to the characteristic $0$ case.
\end{proof}

\begin{rem} \label{rem:Nygaard p=2}
We cited a theorem of Nygaard \cite{Nygaard:higherdeRham-Witt}*{Theorem 2.1}, which is stated for $p \neq 2$.
We show that this is still true if $p = 2$: this argument is due to Kazuhiro Ito.
The only step where the assumption $p = 2$ is used is 
the inductive step of \cite{Ogus:K3crystals}*{Lemma 3.14}. 
If $p = 2$, we can argue as follows.
If there exists $x \in \Gamma$ with $p \notdivides \pairing{x}{x}$ then we argue as in \cite{Ogus:K3crystals}.
Suppose there is no such $x$. There are still $x_1,x_2 \in \Gamma$ with $p \notdivides \pairing{x_1}{x_2}$, 
and then the matrix $(\pairing{x_i}{x_j})_{i,j=1}^2$ is invertible (since $p \divides \pairing{x_i}{x_i}$),
hence $\Gamma$ decomposes to the sum of the subspace generated by $(x_1,x_2)$ and its orthogonal complement. Apply the induction hypothesis to the complement.
\end{rem}

We recall the transcendental lattices of K3 surfaces.
Let $T_l(X)$ be the orthogonal complement of $\NS(X_{\overline k}) \otimes \bZ_l(-1)$ in $\Het^2(X_{\overline k}, \bZ_l)$
and denote by $\chi_l \colon \Aut(X) \to \GL (T_l(X))$ the natural action.
If $k = \bC$ we define $T(X) \subset H^2(X, \bZ)$ and $\chi \colon \Aut(X) \to \GL (T(X))$ similarly.
If $\charac k > 0$ 
we define $T_\crys(X) \subset \Hcrys^2(X_{\overline{k}}/W(\overline k))$ and $\chi_\crys \colon \Aut(X) \to \GL (T_\crys(X))$ similarly,
where $W(\overline k)$ is the ring of Witt vectors over $\overline k$.

In characteristic $0$ the following relationship between $\chi$ and $\rho$ is well-known
(for example see \cite{Schutt:dynamicsssK3}*{Remark 3.4}).
We include the proof for the reader's convenience.
\begin{lem} \label{lem:chi:0}
Assume $\charac k = 0$.
Then the characteristic polynomial of $\chi_l(g)$ belongs to $\bZ[x]$ and is independent of $l$.
It is a power of the $m(g)$-th cyclotomic polynomial $\Phi_{m(g)}$.
If $k = \bC$ then it is equal to that of $\chi(g)$.
\end{lem}

\begin{proof}
We may assume $k = \bC$.
By the comparison of Betti and \'etale cohomology groups,
$\chi(g)$ and $\chi_l(g)$ have the same characteristic polynomial $P \in \bZ[x]$.
Since $H^0(X, \Omega^2) \subset T(X) \otimes \bC$ (by Hodge decomposition),
$P$ has $\rho(g)$ as a root, and hence is divisible by $\Phi_{m(g)}$.
Note that $T(X)_\bQ$ is irreducible as a rational Hodge structure 
(if it admits a decomposition $T(X)_{\bQ} = T_1 \oplus T_2$, then since $\dim (T(X)_{\bC})^{2,0} = h^{2,0}(X) = 1$ there exists $i \in \set{1,2}$ for which $T_i \subset H^{1,1}$,
but then by the Lefschetz $(1,1)$-theorem we have $T_i \subset H^{1,1} \cap H^2(X, \bQ) = \NS(X)_{\bQ}$, hence $T_i = 0$).
Hence $P$ has no other irreducible factor in $\bZ[x]$.
\end{proof}

We also need a positive characteristic version.
We say that $g \in \Aut(X)$ in positive characteristic is \emph{liftable to characteristic $0$}
if
there exists a pair $(\tilde X, \tilde g)$ of a proper smooth scheme $\tilde X$ over a DVR $V$ that is finite over $W(\overline k)$ 
and an automorphism $\tilde g \in \Aut(\tilde X)$
satisfying $(\tilde X, \tilde g) \otimes_V \overline k = (X,g) \otimes_k \overline k$.
(By \cite{Deligne:relevement}*{proof of Corollaire 1.10}, the generic fiber of $\tilde X$ is then automatically a K3 surface.)

\begin{lem} \label{lem:chi:p}
Assume $\charac k = p > 0$.

\textup{(1)} The characteristic polynomial $P$ of $\chi_l(g)$ belongs to $\bZ[x]$, is independent of $l$,
and is equal to that of $\chi_\crys(g)$.

\textup{(2)} If $g$ is liftable to characteristic $0$,
then 
$P$ is a power of $\Phi_{m(g) p^e}$
for some integer $e \geq 0$.

\textup{(3)} If $p > 2$ and $X$ is of finite height (we no longer assume liftability),
then $P$ is a product of cyclotomic polynomials of the form $\Phi_{m(g) p^{e_i}}$ for some integers $e_i \geq 0$. 
%
%\textup{(4)} & INCORRECT. 20210701
%In the situation of \textup{(3)}, assume moreover $X$ is defined over a finite field.
%Then $P$ is a power of $\Phi_{m(g) p^e}$
%for some integer $e \geq 0$.
\end{lem}

\begin{proof}
We may assume $k$ is algebraically closed.

(1)
This follows from the corresponding assertions for the actions 
on $\Het^2$ and $\Hcrys^2$ (showed in \cite{Illusie:report}*{3.7.3})
and on their subspaces generated by $\NS$ (clear).

(2)
Let $(\tilde X, \tilde g)$ be a lifting over $V$
and let $K = \Frac V$.
Comparing the actions of $\tilde g$ on a 2-form on $\tilde X$ and its mod $p$ reduction,
we observe that $\rho(\tilde g \rvert_{\tilde X_K})$ maps to $\rho(g)$
under the map $\mu_{m}(K) \to \mu_{m}(\overline k)$, where $m = m(\tilde g \rvert_{\tilde X_K})$.
Since the kernel of this map is precisely the elements whose order is a power of $p$,
we obtain $m(\tilde g \rvert_{\tilde X_K}) = m(g) p^e$ for some $e \geq 0$.
Since $\NS(\tilde X_{\overline {K}}) \injto \NS(X_{\overline k})$,
we have an equivariant injection $T_l(X) \injto T_l(\tilde X_{K})$,
and the assertion follows from Lemma \ref{lem:chi:0}.

(3)
The images of $\chi_l$ and $\chi_\crys$ are finite 
(this can be showed by reducing to characteristic $0$ as in the proof of Lemma \ref{lem:rho}).
By replacing $g$ with its ${p^N}$-th power for some $N$
we may assume
that the order of $\chi_\crys(g)$ is prime to $p$;
this does not change $m(g)$ because $p \notdivides m(g)$ (since there are no primitive $p$-th roots of $1$ in $k$).
By \cite{Jang:lifting}*{Theorem 3.2}, an automorphism of a K3 surface of finite height in characteristic $p > 2$ with this property
is liftable to characteristic $0$.
Hence the assertion is reduced to (2).
%
%(4) % INCORRECT. 20210701
%Assume $X$ and $g$ are defined over $k = \bF_q$, $q = p^a$.
%By extending $k$ we may assume that $\NS(X_{\overline k}) = \NS(X)$.
%Consider the $F$-isocrystal $H^2 = \Hcrys^2(X/K)$, where $K = \Frac W(k)$, 
%and its $F$-subisocrystal $T = (\NS(X) \otimes K)^\perp$
%(so $T \otimes_{K} \Frac W(\overline k) = T_\crys(X) \otimes_{W(\overline k)} \Frac W(\overline k)$).
%Let $H^2_{<1}$, $H^2_{=1}$, and $H^2_{>1}$ be the $F$-subisocrystals of $T$ with indicated slopes.
%We first show that $T \cap H^2_{=1} = 0$.
%The $a$-th iterate $F^a$ of $F$ is $K$-linear (not only semilinear).
%The eigenvalues of $F^a$ on $H^2_{=1}$ are all $q$ times a roots of unity (since $H^2_{=1}$ is of constant slope).
%By extending $k$ we may assume they are all $q$.
%Then by Tate conjecture $H^2_{=1} \subset \NS(X) \otimes K$. Hence $T \cap H^2_{=1} = 0$.
%
%Take a decomposition $P = P_1 P_2$ into coprime polynomials $P_1, P_2$.
%We shall show $P_1 = 1$ or $P_2 = 1$.
%Let $T_i = T^{P_i(g) = 0}$. Then we have an orthogonal decomposition $T = T_1 \oplus T_2$ into $F$-subisocrystals.
%Then, since $H^2_{<1}$ and $H^2_{>1}$ are irreducible, 
%$H^2_{<1}$ is contained in $T_1$ or $T_2$.
%We may assume $H^2_{<1} \subset T_1$.
%Since $T_2 = T_1^{\perp} \subset (H^2_{<1})^{\perp}$, 
%$H^2_{>1}$ is also contained in $T_1$.
%Hence $T_2 \subset H^2_{=1}$.
%Since $T \cap H^2_{=1} = 0$, we have $T_2 = 0$ and hence $P_2 = 1$.
\end{proof}

In general $e$ and $e_i$ in (2),(3) may be nonzero: see Example \ref{ex:nonzero}.
We do not know whether $P$ in (3) can have more than one different factors.

\section{Action of correspondences on the weight spectral sequence} \label{sec:wss}

Let $l$ be an arbitrary prime different from $\charac k$.
In this section we study the actions 
of automorphisms, and more generally of algebraic correspondences,
on the $l$-adic cohomology groups of varieties $X$ over the fraction field $K$ of a Henselian DVR $\cO_K$.
We show that they act on the graded quotients $\gr^W_n$ of the weight filtration 
and that for certain $n$ these actions are rational,
i.e.\ their characteristic polynomials have coefficients in $\bQ$.

In this paper, we call an algebraic space $\cX$ over $\cO_K$ 
to be a \emph{strictly semistable model} of its generic fiber $X$
if it is regular and flat over $\cO_K$, 
its generic fiber $X$ over $K$ is a smooth \emph{scheme},
and its special fiber $X_0$ over $k$ is a simple normal crossing divisor that is a \emph{scheme}.
($\cX$ itself is not assumed to be a scheme.)

We review the following results on the \emph{weight spectral sequence}.

\begin{thm}
Let $\cX$ be a strictly semistable model over $\cO_K$ of a variety $X$.
Then one can attach to $\cX$ a spectral sequence
\[
E_1^{p,q} 
= \bigoplus_{i \geq \max\{0,-p\}} \Het^{q-2i} (X_{\overline 0}^{(p+2i)}, \bQ_l(-i))
\Rightarrow \Het^{p+q}({X}_{\overline K}, \bQ_l) 
,\]
where $X_{\overline 0}^{(p)}$ are the disjoint unions of $(p+1)$-fold intersections 
of the irreducible components of $X_{\overline 0} := X_0 \otimes_k \overline{k}$.
The spectral sequence is compatible with automorphisms of $\cX$.
The spectral sequence degenerates at $E_2$.
\end{thm}
Under the assumption that $\cX$ is a scheme,
the first assertion is proved by Rapoport--Zink \cite{Rapoport--Zink:monodromie}*{Satz 2.10} and also by Saito \cite{Saito:weightSS}*{Corollary 2.8},
and $E_2$-degeneration is proved by Nakayama \cite{Nakayama:degeneration}*{Proposition 1.9}.
Compatibility with automorphisms follows from the proof of Saito.
In \cite{Matsumoto:goodreductionK3}*{Proposition 2.3} we removed the assumption that $\cX$ is a scheme.

The filtration induced by this spectral sequence is called the \emph{weight filtration} and denoted $W_n \Het^i(X_{\overline K}, \bQ_l)$.
This filtration is independent of the choice of a proper strictly semistable model $\cX$ 
(use the argument in \cite{Ito:weight-monodromy-equal}*{Section 2.3}).
We let $\gr^W_n := \gr^W_n \Het^i = W_n \Het^i / W_{n-1} \Het^i$.

More generally,
we can define the weight filtration on $\Het^i(X_{\overline K}, \bQ_l)$ in the following way
without assuming the existence of a proper strictly semistable model of $X$.
It is known \cite{deJong:alteration}*{Theorem 6.5} that, after replacing $\cO_K$ by a finite extension of its completion, there exists an alteration $Y$ of $X$
(i.e.\ a proper surjective generically-finite morphism $f \colon Y \to X$ from a variety $Y$)
that admits a proper strictly semistable model that is a scheme.
The composite $f_* \circ f^* \colon \Het^i(X_{\overline K}, \bQ_l) \to \Het^i(Y_{\overline K}, \bQ_l) \to \Het^i(X_{\overline K}, \bQ_l)$
is equal to the multiplication by $\deg(f)$,
and we regard $\Het^i(X_{\overline K}, \bQ_l)$ to be a direct summand of $\Het^i(Y_{\overline K}, \bQ_l)$.
We define $W_n \Het^i(X_{\overline K}, \bQ_l)$ to be 
the restriction of $W_n \Het^i(Y_{\overline K}, \bQ_l)$ 
(the weight filtration of $\Het^i(Y_{\overline K}, \bQ_l)$ defined by using a proper strictly semistable model of $Y$).
Again using the argument in \cite{Ito:weight-monodromy-equal}*{Section 2.3},
we can show that this filtration is independent of the choice of $Y$ and a strictly semistable model of $Y$.
In particular, if $X$ itself admits a strictly semistable model,
then this filtration coincides with the one defined in the previous paragraph.

We also recall the \emph{monodromy filtration} $M_r \Het^i$ on $\Het^i = \Het^i(X_{\overline K}, \bQ_l)$
(see \cite{Saito:weightSS}*{Section 2.1} for details).
Let $t_l \colon I_K \to \bZ_l(1)$ be the homomorphism
defined by $t_l(\sigma) = (\sigma(\pi^{1/l^n}) / \pi^{1/l^n})_{n \in \bZ_{\geq 0}}$,
where $\pi$ is a uniformizer of $\cO_K$,
$(\pi^{1/l^n})_{n \in \bZ_{\geq 0}}$ is a system of $l^{n}$-th roots of $\pi$, and
$I_K = \Ker ( \Gal(\overline K/K) \to \Gal(\overline k/k) )$ is the inertia subgroup.
The monodromy operator $N$ is the unique nilpotent map $N \colon \Het^i(1) \to \Het^i$
such that there exists an open subgroup $J \subset I_K$
such that any element $\sigma \in J$ acts on $\Het^i$ by $\exp(t_l(\sigma) N)$.
Then the monodromy filtration is defined as the unique increasing filtration satisfying 
$M_r = 0$ for $r \ll 0$, $M_r = \Het^i$ for $r \gg 0$,
$N(M_r) \subset M_{r-2}$, and 
$N^r \colon \gr^M_r \isomto \gr^M_{-r}$.
An equivalent definition is $M_r \Het^i = \sum_{p,q \in \bZ_{\geq 0}, p - q = r} (\Ker N^{p+1} \cap \Image N^q)$.
In particular, since $N$ acts by zero on the classes of algebraic cycles, 
we have $\NS(X) \subset \Ker N \subset M_0 \Het^2(1)$.

The \emph{weight-monodromy conjecture} states that 
$M_r \Het^i = W_{r+i} \Het^i$ for any $X$ and any $i$.

\begin{thm} \label{thm:weight-monodromy conjecture}
If $\dim X = 2$ then the weight-monodromy conjecture for $\Het^2$ is true,
i.e.\ we have $M_r \Het^2 = W_{r+2} \Het^2$. 
\end{thm}
This is proved by Rapoport--Zink \cite{Rapoport--Zink:monodromie}*{Satz 2.13} in the case $X$ admits a strictly semistable model that is a scheme.
The general case is reduced to this case by de Jong's alteration (see Saito \cite{Saito:weightSS}*{Lemma 3.9}).

Now we consider the actions of algebraic correspondences 
on the $l$-adic cohomology groups.

\begin{thm} \label{thm:action on gr}
Let $X$ be a proper smooth variety over $K$ of dimension $d$,
and $\Gamma \in \CH^d(X \times X)$ an algebraic correspondence
(i.e.\ $\Gamma$ is a $\bZ$-linear combination of codimension $d$ subvarieties of $X \times X$).
Then,

\textup{(1)} For each integer $0 \leq i \leq 2d$,
the action of $\Gamma^*$ on $\Het^i(X_{\overline K}, \bQ_l)$ preserves the weight filtration.
Hence it acts on $\gr^W_n = \gr_n^W \Het^i(X_{\overline K}, \bQ_l)$. 

\textup{(2)} For each integer $0 \leq i \leq 2d$ and each $n \in \set{ 0, 1, 2d-1, 2d}$, 
the characteristic polynomial of $\Gamma^* \rvert_{\gr^W_n}$
is in $\bZ[x]$ and independent of $l$.
If $d \leq 2$, then this holds for all $0 \leq n \leq 2d$.
\end{thm}

To the author's knowledge,
this result is not previously known.

We use the following lemma, which follows from \cite{Kleiman:weilconjectures}*{Lemma 2.8} (cf.\ \cite{Saito:weightSS}*{Lemma 3.4}).
\begin{lem} \label{lem:denominator}
Let $F$ be a field of characteristic $0$
and $f$ an $F$-linear endomorphism of an $F$-vector space of finite dimension.
If there exists a nonzero integer $N$ such that any power of $f$ has trace in $N^{-1} \bZ$,
then the characteristic polynomial of $f$ has coefficients in $\bZ$.
\end{lem}

\begin{proof}[Proof of Theorem \ref{thm:action on gr}]
First we prove the assertions (1) and (2) under
the assumption that $X$ admits a proper strictly semistable model $\cX$ \emph{that is a scheme}.

By \cite{Saito:weightSS}*{Proposition 2.20} there exists 
a collection of algebraic cycles $\overline \Gamma^{(p)} \in \CH^{d-p}(X_{\overline 0}^{(p)} \times X_{\overline 0}^{(p)})$ ($p \geq 0$)
such that
there is an endomorphism of the weight spectral sequence that acts on 
\[
{E}_1^{p,q} = \bigoplus_{i \geq \max\{0,-p\}} \Het^{q-2i} (X_{\overline 0}^{(p+2i)}, \bQ_l(-i))
\quad \text{by} \quad
(d!)^{-1} \cdot \bigoplus_{i \geq \max\{0,-p\}} {{\overline{\Gamma}}^{(p+2i)}}^*
\]
and on $\Het^{p+q}({X}_{\overline K}, \bQ_l)$ by $\Gamma^*$.
Hence (1) follows.

Next we show (2).

Assume $n = 0$.
Then $\gr^W_0 \Het^i = E_2^{i,0}$ is the $i$-th cohomology of the complex 
$E_1^{\bullet, 0} = \Het^0(X_{\overline 0}^{(\bullet)}, \bQ_l)$.
This complex is naturally isomorphic to the base change $V^{(\bullet)} \otimes_\bZ \bQ_l$ 
of the complex $V^{(\bullet)}$
of the $\bZ$-modules freely generated by the components of $X_{\overline 0}^{(\bullet)}$.
The algebraic correspondences $\overline \Gamma^{(p)}$ 
induce an endomorphism of this complex of $\bZ$-modules, 
independent of $l$.
Hence the coefficients of the characteristic polynomial of $\Gamma^* \restrictedto{\gr^W_0}$ lie in $(d!)^{-1} \bZ$,
and the same assertion holds for any power of $\Gamma$.
Hence by Lemma \ref{lem:denominator} the coefficients actually lie in $\bZ$. 

Assume $n = 1$.
Then $\gr^W_1 \Het^i = E_2^{i-1,1}$ is
the $i$-th cohomology of the complex $E_1^{\bullet-1,1} = \Het^1(X_{\overline 0}^{(\bullet - 1)}, \bQ_l)$.
By \cite{Saito:weightSS}*{Lemma 3.6}, 
this complex is naturally isomorphic to $T_l(A^{(\bullet-1)}) \otimes \bQ_l(-1)$
induced by the complex $A^{(\bullet-1)}$ of the Picard varieties of $X_{\overline 0}^{(\bullet-1)}$.
As above,
$\overline \Gamma^{(p)}$ induce 
an endomorphism of this complex of abelian varieties
that is independent of $l$.
Hence the action of $\Gamma$ on $\gr^W_1 \Het^i$ is $(d!)^{-1}$ times the $l$-adic realization of an endomorphism of an abelian variety,
and therefore its characteristic polynomial has coefficients in $(d!)^{-1} \bZ$ and is independent of $l$.
Again by Lemma \ref{lem:denominator} the coefficients lie in $\bZ$.

The cases of $n = 2d-1, 2d$ are similar to the cases of $n = 1, 0$ respectively.

Consider the latter assertion of (2).
Since $\gr^W_n$ is $0$ outside the range $0 \leq n \leq 2d$, 
we are done if $d \leq 1$.
If $d = 2$, then the only remaining case of $n = 2$ 
follows from the assertions for $\Het^i$ (\cite{Saito:weightSS}*{Corollary 0.2})
and for $\gr^W_n$ for the other $n$'s ($n = 0,1,3,4$).

Now we show that the assertions for the general case (i.e.\ not assuming the existence of a semistable model that is a scheme)
can be reduced to this special case.
Take an alteration $Y \to X$ as above (after replacing $\cO_K$ if necessary).
Assertion (1) applied to $f^* \circ f_* \in \CH^d(Y \times Y)$ implies that 
$f_*(W_n \Het^i(Y_{\overline K}, \bQ_l))$ is contained in $W_n \Het^i(X_{\overline K}, \bQ_l)$.
Hence $\Gamma^* = \deg (f)^{-2} \cdot f_* \circ (f^* \Gamma)^* \circ f^*$ preserves the filtration.
Using Lemma \ref{lem:denominator} and the equality 
$\trace(\Gamma^* \mid W_n \Het^i(X_{\overline K}, \bQ_l)) 
= \deg(f)^{-1} \cdot \trace(f^* \circ \Gamma^* \circ f_* \mid W_n \Het^i(Y_{\overline K}, \bQ_l))$
for $\Gamma$ and its powers,
we reduce assertion (2) for $\Gamma$ to the case of $Y$.
\end{proof}

For $n = 1, 2d-1$ we actually proved:
\begin{cor} \label{cor:action on gr1}
Let $X$ and $\Gamma$ be as in Theorem \ref{thm:action on gr}.
Then for each integer $0 \leq i \leq 2d$ and each $n \in \set{ 1, 2d-1}$, 
$\Gamma^* \rvert_{\gr^W_n} \in \End(\gr^W_n \Het^i)$
lies in the image of the algebra $\End(B_n^{(i)}) \otimes \bQ$,
where $B_n^{(i)}$ is an abelian variety with $\gr^W_n \Het^i \cong \Het^1((B_n^{(i)})_{\overline k}, \bQ_l)$.
\end{cor}
$B_1^{(i)}$ is obtained as the $i$-th ``cohomology'' of the complex $A^{(\bullet-1)}$ above,
and $B_{n-1}^{(i)}$ is the dual of $B_1^{(i)}$.

\begin{rem}
We can also show Theorem \ref{thm:action on gr}
without using alterations
if $X$ admits a strictly semistable model that is not necessarily a scheme,
which is the case in the setting of our main theorems.
For this,
we need to generalize \cite{Saito:weightSS}*{Proposition 2.20} to the case of algebraic spaces.
Most of its proof is \'etale-local and hence can be reduced to the scheme case,
and the only non-trivial part is the construction (\cite{Saito:weightSS}*{Lemma 2.17}) of cycle classes $\Gamma^{(p)}$ on $X_{\overline 0}^{(p)}$
for a cycle class $\Gamma$ on the generic fiber $X$.
Although we cannot mimic the construction of Saito (which uses locally-free resolution of coherent sheaves),
we can take an algebraic cycle on $X_{\overline 0}^{(p)}$ to be the intersection with the closure of (a cycle representing) $\Gamma$,
and then check the required properties \'etale-locally.
We omit the details.
\end{rem}

\section{Kulikov models} \label{sec:Kulikov}

\begin{defn}
Let $X$ be a K3 surface or an abelian surface over $K$.
An algebraic space $\cX$ over $\cO_K$ with generic fiber $X$ is a \emph{Kulikov model} of $X$
if it is a proper strictly semistable model (in the sense of the previous section)
and its relative canonical divisor $K_{\cX/\cO_K}$ is trivial.
\end{defn}

\begin{rem} \label{rem:Kulikov}
The standard, but conditional, recipe to construct a Kulikov model of a K3 (or an abelian) surface $X$
is the following. 
Take a proper strictly semistable model of $X$ after extending $K$,
then apply a suitable MMP to get a log terminal model with nef canonical divisor (then the canonical divisor is in fact trivial),
and then apply Artin's simultaneous resolution to make it semistable.
If the residue field $k$ is of characteristic $0$ this is unconditional.
If $\charac k = p > 0$ then the existence of a semistable model is still conjectural
and the MMP is proved only for $p \geq 5$ \cite{Kawamata:mixed3fold}.
\end{rem}

Classification of the special fibers of Kulikov models is given by 
Kulikov \cite{Kulikov:degeneration}*{Theorem II} in characteristic $0$
and Nakkajima \cite{Nakkajima:logk3}*{Proposition 3.4} in characteristic $>0$.
Then we can compute $\dim \gr^W_n$ using the weight spectral sequence.
We summarize:
\begin{prop} \label{prop:Kulikov:K3}
Let $X$ be a K3 surface and $\cX$ a Kulikov model of $X$.

\textup{(1)} The geometric special fiber $X_{\overline 0} = X_0 \otimes \overline k$ is one of the following.
\begin{description}
 \item [Type I] a smooth K3 surface.
 \item [Type II] $X_{\overline 0} = Z_1 \cup \cdots \cup Z_N$, $N \geq 2$,
 $Z_1$ and $Z_N$ are rational and the others are elliptic ruled (i.e.\ birational to a $\bP^1$-bundle over an elliptic curve),
 $Z_i \cap Z_j$ is an elliptic curve if $\lvert i-j \rvert = 1$, and there are no other intersections.
 (The components form a ``chain''.)
 \item [Type III] $X_{\overline 0}$ is a union of rational surfaces, each double curve is rational,
 and the dual graph of the components forms a triangulation of $S^2$ (the $2$-dimensional sphere).
\end{description}

\textup{(2)} $\dim_{\bQ_l} \gr^W_n \Het^2(X_{\overline K}, \bQ_l)$ ($n = 0,1,2,3,4$) of the generic fiber $X$ depends only on the type,
and are given by the following.
\begin{description}
 \item [Type I]   $0,0,22,0,0$.
 \item [Type II]  $0,2,18,2,0$.
 \item [Type III] $1,0,20,0,1$.
\end{description}
\end{prop}

We will simply say that $X_0$ is of Type I, II, or III if $X_{\overline 0}$ is so.

\section{Proof of the main theorems} \label{sec:proof}

\begin{lem} \label{lem:action on gr}
Let $K$, $X$, and $\cX$ be as in Theorem \ref{thm:main} (so $\charac K \neq 2$).
Assume either $\charac K = p > 2$ and $\height(X) < \infty$, or $\charac K = 0$.
Assume $X_0$ is of Type II or III, and let $n = 3$ or $n = 4$ respectively.
Then any eigenvalue of the action of $g \in \Aut(X)$ on $\gr^W_n = \gr_n^W \Het^2(X_{\overline K}, \bQ_l)$ 
is a primitive $m'$-th root of $1$,
where $m' = m(g)p^e$ for some integer $e \geq 0$ (possibly depending on $g$ and the eigenvalue) if $\charac K = p > 0$,
and $m' = m(g)$ if $\charac K = 0$.
\end{lem}
\begin{proof}
By Proposition \ref{prop:Kulikov:K3} we have $\gr^W_n = W_n / W_2 \neq 0$.

As noted in Section \ref{sec:wss},
we have $\NS(X_{\overline K}) \otimes_\bZ \bQ_l(-1) \subset \Ker N \subset M_0 \Het^2(X_{\overline K}, \bQ_l)$,
and we have $M_0 \Het^2 = W_2 \Het^2$ (Theorem \ref{thm:weight-monodromy conjecture}).
Hence $\gr^W_n = W_n / W_2$ is a quotient of $T_l(X) \otimes_{\bZ_l} \bQ_l$.
The assertions follow from Lemmas \ref{lem:chi:0}--\ref{lem:chi:p}.
\end{proof}

\begin{proof}[Proof of Theorem \ref{thm:main}]
If $\charac K > 0$ and $\height(X) \geq 3$,
then it is proved in \cite{Rudakov--Zink--Shafarevich}*{Section 2, Corollary}
(without the condition on automorphisms) that $X_0$ is of Type I:
Since the height is upper semi-continuous
and Type II (resp.\ III) surfaces have height $\leq 2$ (resp.\ $= 1$),
$X_0$ cannot be of Type II nor III.

Now assume either $\charac K = p > 2$ and $\height(X) < \infty$, or $\charac K = 0$.
We can apply Lemma \ref{lem:action on gr}
to $\psi_l \colon \Aut(X) \to \GL(\gr^W_n)$.

Assume that $X_0$ is of Type II. Let $n = 3$.
By Corollary \ref{cor:action on gr1}, 
$\psi_l$ factors through $(\End(C)_\bQ)^*$, 
where $C$ is the elliptic curve appearing as the intersection of two components of $X_{\overline 0}$.
For any $g \in \Aut(X)$,
$\psi_l(g)$ belongs to a (commutative) $\bQ$-subalgebra of $\End(C)_\bQ$ generated by a single element ($\psi_l(g)$ itself),
and such a subalgebra is either $\bQ$ or an imaginary quadratic field.
Hence $m' \in \{ 1,2,3,4,6 \}$ and so is $m(g)$.

Assume that $X_0$ is of Type III. Let $n = 4$.
Similarly by Theorem \ref{thm:action on gr} we have $\psi_l(g) \in \GL_1(\bQ) = \bQ^*$.
Hence $m' \in \{ 1,2 \}$ and so is $m(g)$.
\end{proof}

\begin{cor} \label{cor:orientation}
Assume that $X_0$ is of Type III.
Let $g \in \Aut(X)$,
and suppose $g$ extends to an automorphism of $\cX$,
so that $g$ acts on $X_0$ and on the set of the irreducible components of $X_{\overline 0}$.
Then the induced action of $g$ on
the ($2$-element) set of orientations of $S^2$ (which the dual graph of $X_{\overline 0}$ triangulates)
coincides with the image $\rho(g) \in \set{\pm 1}$ of $g$ by $\rho \colon \Aut(X) \to \GL (H^0(X, \Omega^2))$.
\end{cor}
\begin{proof}
Indeed, we have 
\begin{align*}
\gr^W_4 \Het^2 = E_2^{-2,4} &= \Ker (\Het^0(X_{\overline 0}^{(2)}, \bQ_l(-2)) \to \Het^2(X_{\overline 0}^{(1)}, \bQ_l(-1))) \\
&\cong H_2(S^2, \bQ) \otimes_\bQ \bQ_l(-2),
\end{align*}
and the two generators of $H_2(S^2, \bZ)$ corresponds to the two orientations.
\end{proof}

\begin{rem} \label{rem:orientation}
Actually it was this assertion (Corollary \ref{cor:orientation}), proposed by Yuji Odaka,
that led the author to the study of this paper.
The author first looked for an example of a Type III degeneration with $m(X) = \card{\Image \rho} \geq 3$,
which would be a counterexample to the assertion, but failed to find one.
It turned out that such examples do not exist!
\end{rem}

Next we prove Theorem \ref{thm:E-version}.
First we define the Hodge endomorphism field.

Let $X$ be a K3 surface over a field $F$ of characteristic $0$.
Let $X_\bC$ be a K3 surface over $\bC$ 
isomorphic to $X$ over some field $F'$ containing both $F$ and $\bC$
(such $X_\bC$ always exists).
We call $E = \End_{\HS}(T(X_\bC)_\bQ)$
the \emph{Hodge endomorphism field} of $X$,
where $\End_{\HS}$ denotes the endomorphisms of a rational Hodge structure.
It is known that $E$ is either a totally real field or a CM field (\cite{Zarhin:HodgegroupsK3}*{Theorem 1.5.1}).

This definition of $E$ depends a priori on the choice of $X_\bC$.
However, if the Hodge conjecture for the self-product of a K3 surface holds,
then every element of $E$ can be realized as the action of an algebraic cycle on $X_\bC \times X_\bC$,
hence of an algebraic cycle on $(X \times X)_{\overline F}$,
and therefore $E$ (up to isomorphism) depends only on $X$.
The conjecture for $X_\bC \times X_\bC$ is proved by Ram\'{o}n Mar\'{i} \cite{RamonMari:Hodgeconjecture}*{Theorem 5.4}
under the assumption that $E$ (of $X_\bC$) is a CM field.
Thus, for a CM field $E$, the statement ``$E$ is the Hodge endomorphism field of $X$'' is well-defined,
and the statement of Theorem \ref{thm:E-version} should be understood accordingly.

\begin{proof}[Proof of Theorem \ref{thm:E-version}]
Assume Type II or III, and let $n = 3$ or $n = 4$ respectively.
We have two homomorphisms of $\bQ$-algebras
$\psi_l \colon \CH^2(X \times X)_\bQ \to \End (\gr^W_n) $
and
$\chi \colon \CH^2(X \times X)_\bQ \to \End_{\HS} T(X_\bC)_\bQ \cong E$.
Since we are assuming that
the Hodge conjecture for $X \times X$ is true, 
the $E$-action is realized as algebraic correspondences,
i.e.\  $\chi$ is surjective.
Since $T(X_\bC) \otimes \bQ_l \to \gr^W_n$ is surjective,
we have a surjection $\Image \chi \surjto \Image \psi_l$.
Since $\Image \chi \cong E$ is a field and $\Image \psi_l$ is nonzero, this surjection is an isomorphism.

Assume that $X_0$ is of Type II. Let $n = 3$.
By Corollary \ref{cor:action on gr1},
$\Image \psi_l$ is contained in $\End(C)_\bQ$ for some elliptic curve $C$.
A (commutative) subfield of $\End(C)_\bQ$ is either $\bQ$ or an imaginary quadratic field.

Assume that $X_0$ is of Type III. Let $n = 4$.
Similarly by Theorem \ref{thm:action on gr},
$\Image \psi_l$ is contained in $M_1(\bQ) = \bQ$.
\end{proof}

\section{Application: moduli spaces of K3 surfaces with non-symplectic automorphisms of prime order} \label{sec:application}

We apply the main theorem to obtain a compactification of
the moduli spaces of K3 surfaces with non-symplectic automorphisms of fixed prime order $\geq 5$.

For a moment we work over $\bC$.
In this section a \emph{lattice} is a free $\bZ$-module of finite rank equipped with a $\bZ$-valued symmetric bilinear form.
Let $U$ be the hyperbolic plane (i.e.\ the rank $2$ lattice with Gram matrix 
$\begin{pmatrix} 0 & 1 \\ 1 & 0 \end{pmatrix}$)
and $E_8$ the (negative definite) root lattice of type $E_8$.
Then $\LK3 := U^{\oplus 3} \oplus E_8^{\oplus 2}$ is isometric to $H^2(X, \bZ)$ for any K3 surface $X$ over $\bC$,
and is called the K3 lattice.

We recall the notation of \cite{Artebani--Sarti--Taki}.
Fix a prime $p \leq 19$
and a primitive $p$-th root $\zeta_p$ of $1$.
Fix an isometry $\sigma \in \Ortho(\LK3)$ of order $p$,
and denote by $[\sigma]$ its conjugacy class.
We write $S(\sigma) = (\LK3)^{\sigma = 1}$.
Let $\cM^\sigma$ be the moduli space of a pair $(X,g)$
consisting of a complex K3 surface $X$ 
and an automorphism $g$ of $X$ of order $p$
with $\rho(g) = \zeta_p$ 
and acting on $H^2(X, \bZ)$ by $\sigma$ via some marking (i.e.\ isometry) $\LK3 \isomto H^2(X, \bZ)$.
We call such $(X,g)$ to be a \emph{$[\sigma]$-polarized K3 surface}.
Such $X$ is automatically algebraic and has an ample class in $S(g) = H^2(X, \bZ)^{g = 1}$.
The marking induces an isometry $S(\sigma) \isom S(g)$.

Let
$D^\sigma  = \{ w \in \bP((\LK3 \otimes \bC)^{\sigma = \zeta_p}) \colon (w, w) = 0, (w, \overline{w}) > 0 \}$:
this is a type IV Hermitian symmetric space if $p = 2$
and a complex ball if $p > 2$.
Define the divisor $\Delta^\sigma = \bigcup_{\delta \in (S(\sigma))^\perp, \delta^2 = -2} (D^\sigma \cap \delta^\perp)$
and the discrete group $\Gamma^\sigma = \{ \gamma \in \Ortho(\LK3) \colon \gamma \sigma = \sigma \gamma \}$.
Then the space $\Gamma^\sigma \backslash (D^\sigma \setminus \Delta^\sigma)$
is naturally isomorphic to the space $\cM^\sigma(\bC)$ of $\bC$-valued points of $\cM^\sigma$
(\cite{Artebani--Sarti--Taki}*{Theorem 9.1}).

Hereafter we consider only isometries $\sigma$ for which $\cM^\sigma$ is nonempty.

In fact, the moduli space $\cM^\sigma$ can be defined algebraically over $\bQ(\zeta_p)$.
Indeed, by \cite{Artebani--Sarti--Taki}*{Proposition 9.3},
if $p \geq 3$, then $[\sigma]$ is uniquely determined by the topology of $\Fix(g) \subset X$,
and if $p = 2$, then $[\sigma]$ is uniquely determined by the
parameters $r,a,\delta$ of the lattice $S(\sigma) \isom S(g)$
which can be read off from the action of $g$ on the $2$-adic \'etale cohomology group $\Het^2(X, \bZ_2)$.

Now we consider the compactification of $\cM^\sigma$.
Theorem \ref{thm:main} does not imply that $\cM^\sigma$ is compact,
since the action on the K3 lattice may change by specialization. 
So we need to attach some $\cM^\tau$'s on the boundary.

Consider the space $ \Gamma^\sigma \backslash D^\sigma $.
The points on the boundary correspond to pairs $(X,g)$ with non-symplectic automorphism of order $p$
but acting on $H^2$ by some $\tau \neq \sigma$.
We can translate this into an algebraic construction of $\overline{\cM^\sigma}$ by attaching $\cM^\tau$'s to $\cM^\sigma$,
and we have $\overline{\cM^\sigma}(\bC) = \Gamma^\sigma \backslash D^\sigma $.
Then we have the following.

\begin{prop}
Let $\sigma$ be an isometry of $\LK3$ of order $p \geq 5$.
Then $\overline{\cM^\sigma}$ is proper.
\end{prop}

\begin{proof}
It suffices to show that any $K$-rational point, $K = \bC((t))$,
extends to an $\cO_K$-rational point, $\cO_K = \bC[[t]]$,
after replacing $K$ by a finite extension $\bC((t^{1/n}))$.
So let $(X,g) \in \cM(K)$ be a pair defined over $K$.
Then by Theorem \ref{thm:main} there exists, after extending $K$, a smooth proper algebraic space $\cX$ over $\cO_K = \bC[[t]]$ with K3 fibers,
and $g$ extends to a birational map $g \colon \cX \rationalto \cX$.
The period map gives an extension of the morphism $\Spec \bC((t)) \to \cM^\sigma$ to $\Spec \bC[[t]] \to \overline{\cM^\sigma}$.
By \cite{Matsumoto:extendability}*{Proposition 2.2}, 
$g \colon \cX \rationalto \cX$ is defined over the complement of a closed subspace of $\cX$ of codimension $\geq 2$,
and the birational map $g_0 := g \rvert_{X_0}$ extends to an automorphism $\tilde g_0$ of $X_0$.
This $\tilde g_0$ also satisfy $\rho(\tilde g_0) = \zeta_p$ 
but the action on $H^2$ is possibly different from $\sigma$.
\end{proof}

\begin{example}
Let $p = 11$.
By \cite{Artebani--Sarti--Taki}*{Theorem 7.3}
there are three non-conjugate isometries $\sigma_1, \sigma_2, \sigma_3 \in \Ortho(\LK3)$ of order $11$,
respectively with 
$S(\sigma_1) \isom U$,
$S(\sigma_2) \isom U(11)$,
$S(\sigma_3) \isom U \oplus A_{10}$,
and the geometric points of $\cM^{\sigma_1}$
consists of elliptic K3 surfaces
$y^2 = x^3 + ax + t^{11} - b$
with an automorphism $g(x,y,t) = (x,y,\zeta_{11}t)$,
parametrized by $\set{(a^3:b^2) \in \bP^1 \colon 4a^3 + 27b^2 \neq 0}$.
This elliptic surface has $1$ singular fiber of type $\rII$ at $t = \infty$
and $22$ of type $\rI_1$ at $4a^3 + 27(t^{11} - b)^2 = 0$,
unless $a = 0$, in which case it has $12$ singular fiber of type $\rII$.
(This is Kodaira's notation of singular fibers of elliptic surfaces, and is not related to the classification of Kulikov models.)
The boundary of the compactification $\overline {\cM^{\sigma_1}}$
consists of one point ($4a^3 + 27b^2 = 0$), and this is $\cM^{\sigma_3}$.
At this point the elliptic surface has $1$ singular fiber of type $\rII$, $11$ of type $\rI_1$, and one of type $\rI_{11}$.
\end{example}

The following is also a direct consequence of Theorem \ref{thm:main}.

\begin{prop} \label{prop:everywheregoodreduction}
Assume the existence of a Kulikov model (after field extension) for any K3 surface defined over a number field.
Take an isometry $\sigma$ with order $p \geq 5$.
Then any K3 surface corresponding to a point of $\cM^\sigma (\overline \bQ)$
has everywhere potential good reduction.
\end{prop}

Previously the author proved (conditionally) everywhere potential good reduction of
K3 surfaces with complex multiplications \cite{Matsumoto:goodreductionK3}*{Theorem 6.3},
which however gives only isolated examples.
To the contrary, Proposition \ref{prop:everywheregoodreduction}
can be applied to positive dimensional family:
for $p = 5,7,11$ there exists $\sigma \in \Ortho(\LK3)$ of order $p$ 
with $\dim \cM^\sigma = 4,2,1$ respectively.

\begin{example}
Again consider a K3 surface of the form $y^2 = x^3 + ax + t^{11} - b$,
this time over defined a number field $K$.
Then by the previous proposition it has potential good reduction at any prime of $K$.
We can also show potential good reduction directly.
If the residue characteristic is equal to $11$ then this is done in \cite{Matsumoto:extendability}*{Example 6.8}.
In other characteristics we proceed as follows.
After extending $K$, we can find $a_i \in \cO_K$ such that 
\[
\cE = (F(x',y') = y'^2 + a_1 x'y' + a_3 y' + x'^3 + a_2 x'^2 + a_4 x' + a_6 = 0) \subset \bP^2_{\cO_K}
\]
is a minimal Weierstrass model over $\cO_K$ of the elliptic curve $(y^2 = x^3 + ax - b)$,
and that the special fiber $E_0 = \cE \otimes_{\cO_K} k$ of $\cE$ is either an elliptic curve or a nodal curve.
Then the special fiber of 
\[
\cX = (F(x',y') + t^{11} = 0) 
\]
is smooth or has one $A_{10}$ singularity, according to $E_0$ being smooth or nodal respectively.
Applying Artin's simultaneous resolution (after extending $K$) we achieve good reduction.
\end{example}

\section{Examples}

\subsection{Theorems \ref{thm:main} and \ref{thm:E-version} are optimal}

The following two examples show that we cannot weaken the assumptions on $m$ and $E$ in Theorems 
\ref{thm:main} and \ref{thm:E-version}.

\begin{example}[Type III] \label{ex:Type III}
Assume $\charac k \neq 2$.
Consider the family $\cX'$ over $\cO_K$ of quartic surfaces given by 
$t(w^4 + x^4 + y^4 + z^4) + wxyz = 0$, 
where $t$ is a uniformizer of $\cO_K$.
This is not a Kulikov model, since it has non-regular points,
but we can perform small blow-ups to obtain a Kulikov model $\cX$.
(For example, we can resolve the singularity at $t = w = x = y^4 + z^4$
by blowing-up either the ideal $(t,w)$ or $(t,x)$ at a neighborhood.)
Then the special fiber is of Type III (whose dual graph is a tetrahedron).
The symmetric group $\fS_4$ acts on $X$ (over $K$) naturally 
and its action on the $2$-forms is given by $\sign: \fS_4 \surjto \{ \pm 1 \}$.
Hence $m(X)$ is divisible by $2$
(and by Theorem \ref{thm:main} we have $m(X) = 2$).
\end{example}

\begin{example}[Type II] \label{ex:Type II}
Again assume $\charac k \neq 2$.
Let $E$ be an imaginary quadratic field.
Let $C_1$ be an elliptic curve with complex multiplication by an order of $E$,
and $C_2$ an elliptic curve with multiplicative reduction.
Let $\cC_i$ be the minimal regular models of $C_i$ over $\cO_K$.
By extending $K$, we may assume that $(C_1)_0$ is smooth and that $(C_2)_0$ has an even number of components.
Let $\iota$ be the multiplication-by-$(-1)$ map on $\cA = \cC_1 \times_{\cO_K} \cC_2$
(then $\Fix(\iota)$ is finite \'etale over $\cO_K$),
and $\cX$ the blow-up of $\cA / \iota$ at the image of $\Fix(\iota)$.
Then $\cX$ is a Kulikov model of $X = \Km(C_1 \times C_2)$ with Type II degeneration,
$T(X_\bC)_\bQ \cong H^1((C_1)_\bC,\bQ) \otimes H^1((C_2)_\bC,\bQ)$ and
$\End_{\HS} T(X)_\bQ \supset E$
(and by Theorem \ref{thm:E-version} we have $\End_{\HS} T(X)_\bQ = E$).

In particular, if we take $C_1$ to be the elliptic curve with an automorphism of order $4$ (resp.\ $6$),
then we have an example of Type II degeneration with $m(X)$ divisible by $4$ (resp.\ $6$).
\end{example}

\subsection{Some automorphisms of K3 surfaces of finite order in positive characteristic}

The following collection of examples shows that in general
$e$ (and $e_i$) in Lemma \ref{lem:chi:p}(2),(3) may be nonzero.

\begin{example} \label{ex:nonzero}
Considering the obvious degree constraint ($\deg \Phi_{p^e} \leq 22$),
$p^e$ with $e \neq 0$ can occur 
in Lemma \ref{lem:chi:p}(2),(3) 
only if $p^e$ belongs to the set
\[ \{ 2^e \; (e \leq 5), \; 3^e \; (e \leq 3), \; 5^e \; (e \leq 2), 7, 11, 13, 17, 19 \} .\]
We give examples for $p^e = 2, 2^2, 3, 5, 7, 11$. All examples are automorphisms of finite order.
For the other cases we do not know whether examples exist.

Let $p^e$ be one of $2,2^2,3,5,7,11$.
Define an integer $n$ as in the table below. 
Let $K = \bQ_p(\zeta_{p^e n})$ if $e = 1$
and $K = \bQ_p(\zeta_{3 p^e n})$ if $p^e = 2^2$.
If $p^e = 2,2^2,3,7,11$, 
let $\cX$ (over $\cO_K$) be the elliptic K3 surface defined by the equation below
(with the $E_8$ singularity at $t = \infty$ resolved in the standard way if $p^e = 2,2^2,7$).
If $p^e = 5$,
let $\cX$ (over $\cO_K$) be the blow-up of the double sextic surface defined by the equation below
at the non-smooth locus $(w = x = G_5(y) = 0)$.
Define $g_1,g_2 \in \Aut(\cX)$ as below.

\begin{tabular}{llllll}
\toprule
$p^e$ & $n$ & equation & $g_1(x,y,t)$ & $g_2(x,y,t)$ & char.poly. 
\\ 
\midrule
$2$ & $21$ & 
$G_2(y) = x^3 + t^7 $ &
$(x, -y + 1, t)$ &
$(\zeta_{21}^{7} x, y, \zeta_{21}^{15} t)$ &
$\Phi_1^{10} \Phi_{42}$
\\
$2^2$ & $7$ & 
$H(x,y) + t^7 = 0$ &
$(h(x,y), -t)$ &
$(x, y, \zeta_{7} t)$ &
$\Phi_1^{10} \Phi_{28}$
\\
$3$ & $22$ & 
$y^2 = G_3(x) + t^{11}$ &
$(\zeta_3 x + 1, y, t)$ &
$(x, -y, \zeta_{22}^{12} t)$ & 
$\Phi_1^2 \Phi_{66}$
\\
$5$ & $8$ & 
$w^2 = x (x^4 + G_5(y))$ &
$(x, \zeta_5 y + 1, w)$ &
$(\zeta_8^2 x, y, \zeta_8 w)$ &
$\Phi_1^2 \Phi_5 \Phi_{40}$
\\
$7$ & $6$ & 
$y^2 = x^3 + G_7(t)$ & 
$(x, y, \zeta_7 t + 1)$ &
$(\zeta_6^4 x, - y, t)$ & 
$\Phi_1^{10} \Phi_{42}$
\\
$11$ & $1$ &
$y^2 = x^3 + x^2 + G_{11}(t)$ &
$(x, y, \zeta_{11} t + 1)$ &
$\id$ &
$\Phi_1^{2} \Phi_{11}^2$
\\
\bottomrule
\end{tabular}

Here $G_p$ and $H$ are defined as follows:
\begin{itemize}
\item $G_p(z) = \prod_{i = 0}^{p-1} (z - a_i) \in \bZ_p[\zeta_p][z]$
with $a_i = (\zeta_p^i - 1) / (\zeta_p - 1)$.
It satisfies $G_p(\zeta_p z + 1) = G_p(z)$
and $G_p(z) \equiv z^p - z \pmod{(\zeta_p - 1)}$ (since $a_i \equiv i$).
\item
$H(x,y) = y^2 + a_1 xy + a_3 y + x^3 + a_2 x^2 + a_4 x + a_6 = 0$
and $(x,y) \mapsto h(x,y) = (- x + b_2, \zeta_4^{-1} y + b_1 x + b_3)$,
$a_i,b_i \in \cO_K$, 
are equations of the N\'eron model of an elliptic curve with an automorphism acting on the $1$-forms by $\zeta_4^{-1}$.
(For example, one can take 
$a_1 = 3 - \sqrt{3}$, 
$a_3 = 2 - \sqrt{3}$,
$a_2 = a_4 = a_6 = 0$,
$b_2 =  - (2 - \sqrt{3})$,
$b_1 = \zeta_{12}(1 - \zeta_{12}) \sqrt{3}$,
$b_3 = \zeta_{12}^{-1}(\zeta_{12}-1)^3$,
with $\sqrt{3} = \zeta_{12} + \zeta_{12}^{-1}$.)
\end{itemize}

Then 
$\cX$ is smooth proper over $\cO_K$ with K3 fibers,
$g_1,g_2 \in \Aut(\cX)$ commute
and are of orders $p^e$ and $n$ respectively, 
and $\rho(g_1 \rvert_{X_K}) = \zeta_{p^e}$ and $\rho(g_2 \rvert_{X_K}) = \zeta_{n}$.
Hence $g := g_1 g_2$ is of order $p^e n$, and $m(g \rvert_{X_K}) = p^e n$, $m(g \rvert_{X_0}) = n$.
Therefore $g$ acts on $T_l(X_K)$ by a power of $\Phi_{p^e n}$,
hence also on $T_l(X_0)$.

We can moreover determine the characteristic polynomial of $g$ on $\Het^2$ completely (although we do not need this).
For $p^e = 2,4,3,7$,
one observes that the sublattice of $\NS(X_K)$
generated by the zero section and the components of the singular fibers
already has rank $22 - \phi(p^e n)$,
and hence $\Het^2$ of $X_K$ is generated up to torsion by these curves and $T_l(X_K)$.
Since the action of $g$ on the classes of those curves are trivial, we obtain the characteristic polynomial $\Phi_1^{22 - \phi(p^e n)} \Phi_{p^e n}$.
For $p^e = 5$,
similarly $\Het^2$ is generated by the pullback of $\cO_{\bP^2}(1)$ and the five exceptional curves and $T_l(X_K)$.
Since $g_1$ and $g_2$ respectively act on the classes of the exceptional curves transitively and trivially,
we obtain the characteristic polynomial $\Phi_1^2 \Phi_5 \Phi_{40}$.
Finally, for $p^e = 11$, it is proved by by Dolgachev--Keum \cite{Dolgachev--Keum:order11}*{Lemma 2.3(i)}
that the characteristic polynomial of an order $11$ automorphism of a K3 surface in characteristic $11$ 
is always $\Phi_1^{2} \Phi_{11}^2$.

It remains to show that $X_0$ is not supersingular.
For the case $p^e = 11$ this is checked in \cite{Schutt:order11}*{Section 3.2}.
Assume $p^e \neq 11$.
By \cite{Nygaard:higherdeRham-Witt}*{Theorem 2.1} (see Remark \ref{rem:Nygaard p=2}),
if $g$ is an automorphism of a supersingular K3 surface $Y$ then $m(g)$ should divide $p^{\sigma_0} + 1$,
where $\sigma_0$ is the Artin invariant of $Y$ (which is a positive integer $\leq 10$).
For the cases $p^e \neq 11$
we observe that no such integer $\sigma_0$ exist,
and hence $X_0$ are not supersingular.
As the reader might have noticed,
the generic fibers of $\cX$ for $p^e = 2, 2^2, 3, 5, 7$ are the
well-known examples of automorphisms $g$ of characteristic zero K3 surfaces with $\ord(g) = \ord(\rho(g)) = 42,28,66,40,42$,
given by Kondo \cite{Kondo:trivially}*{Section 3} (order $28,42,66$) and Machida--Oguiso \cite{Machida--Oguiso}*{Proposition 4 (15)} (order $40$).
(To check this use the equality
$G_p(z) = (\zeta_p - 1)^{-p} (z'^p - 1)$
where $z' = (\zeta_p - 1) z + 1$.)
\end{example}

\begin{example}
Keum \cite{Keum:orders} classified possible finite orders of automorphisms of K3 surfaces over each characteristic $\neq 2,3$.
Since the problem is still open for characteristic $2$ and $3$,
we note some examples in these characteristics, 
although these might be known to experts.

Case $p^e = 2^2$ in the previous example gives an automorphism of order $28$ of a K3 surface in characteristic $2$.
By replacing $7$ with $9$ and $11$ we also obtain automorphisms of order $36$ and $44$.

Case $p^e = 2$ gives order $42$ in characteristic $2$, and by replacing $7$ with $11$ we obtain order $66$.
However these two examples are almost written in \cite{Keum:orders}*{Example 3.6}.

Case $p^e = 3$ gives an automorphism of order $66$ of a K3 surface in characteristic $3$.
By replacing $11$ with $7$, $8$ and $10$ (and then resolving the singularity at $t = \infty$ in the standard way) 
we also obtain automorphisms of order $42$, $48$ and $60$.
\end{example}


\begin{bibdiv}
\begin{biblist}
\bib{Artebani--Sarti--Taki}{article}{
  author={Artebani, Michela},
  author={Sarti, Alessandra},
  author={Taki, Shingo},
  title={$K3$ surfaces with non-symplectic automorphisms of prime order},
  note={With an appendix by Shigeyuki Kond\=o},
  journal={Math. Z.},
  volume={268},
  date={2011},
  number={1-2},
  pages={507--533},
  issn={0025-5874},
}

\bib{deJong:alteration}{article}{
  author={de Jong, A. J.},
  title={Smoothness, semi-stability and alterations},
  journal={Inst. Hautes \'Etudes Sci. Publ. Math.},
  number={83},
  date={1996},
  pages={51--93},
  issn={0073-8301},
}

\bib{Deligne:relevement}{article}{
  author={Deligne, P.},
  title={Rel\`evement des surfaces $K3$ en caract\'eristique nulle},
  language={French},
  note={Prepared for publication by Luc Illusie},
  conference={ title={Algebraic surfaces}, address={Orsay}, date={1976--78}, },
  book={ series={Lecture Notes in Math.}, volume={868}, publisher={Springer}, place={Berlin}, },
  date={1981},
  pages={58--79},
}

\bib{Dolgachev--Keum:order11}{article}{
  author={Dolgachev, Igor V.},
  author={Keum, JongHae},
  title={$K3$ surfaces with a symplectic automorphism of order 11},
  journal={J. Eur. Math. Soc. (JEMS)},
  volume={11},
  date={2009},
  number={4},
  pages={799--818},
  issn={1435-9855},
}

\bib{Illusie:report}{article}{
  author={Illusie, Luc},
  title={Report on crystalline cohomology},
  conference={ title={Algebraic geometry}, address={Proc. Sympos. Pure Math., Vol. 29, Humboldt State Univ., Arcata, Calif.}, date={1974}, },
  book={ publisher={Amer. Math. Soc., Providence, R.I.}, },
  date={1975},
  pages={459--478},
}

\bib{Ito:weight-monodromy-equal}{article}{
  author={Ito, Tetsushi},
  title={Weight-monodromy conjecture over equal characteristic local fields},
  journal={Amer. J. Math.},
  volume={127},
  date={2005},
  number={3},
  pages={647--658},
  issn={0002-9327},
}

\bib{Jang:lifting}{article}{
  author={Jang, Junmyeong},
  title={A lifting of an automorphism of a K3 surface over odd characteristic},
  journal={Int. Math. Res. Notices},
  date={2016},
}

\bib{Kawamata:mixed3fold}{article}{
  author={Kawamata, Yujiro},
  title={Semistable minimal models of threefolds in positive or mixed characteristic},
  journal={J. Algebraic Geom.},
  volume={3},
  date={1994},
  number={3},
  pages={463--491},
  issn={1056-3911},
}

\bib{Keum:orders}{article}{
  author={Keum, JongHae},
  title={Orders of automorphisms of K3 surfaces},
  journal={Adv. Math.},
  volume={303},
  date={2016},
  pages={39--87},
  issn={0001-8708},
}

\bib{Kleiman:weilconjectures}{article}{
  author={Kleiman, S. L.},
  title={Algebraic cycles and the Weil conjectures},
  conference={ title={Dix expos\'es sur la cohomologie des sch\'emas}, },
  book={ series={Adv. Stud. Pure Math.}, volume={3}, publisher={North-Holland, Amsterdam}, },
  date={1968},
  pages={359--386},
}

\bib{Kondo:trivially}{article}{
  author={Kond{\=o}, Shigeyuki},
  title={Automorphisms of algebraic $K3$ surfaces which act trivially on Picard groups},
  journal={J. Math. Soc. Japan},
  volume={44},
  date={1992},
  number={1},
  pages={75--98},
  issn={0025-5645},
}

\bib{Kulikov:degeneration}{article}{
  author={Kulikov, Viktor S.},
  title={Degenerations of $K3$ surfaces and Enriques surfaces},
  language={Russian},
  journal={Izv. Akad. Nauk SSSR Ser. Mat.},
  volume={41},
  date={1977},
  number={5},
  pages={1008--1042, 1199},
}

\bib{Lieblich--Maulik:cone}{article}{
  author={Lieblich, Max},
  author={Maulik, Davesh},
  title={A note on the cone conjecture for K3 surfaces in positive characteristic},
  journal={Math. Res. Lett.},
  volume={25},
  date={2018},
  number={6},
  pages={1879--1891},
  issn={1073-2780},
}

\bib{Machida--Oguiso}{article}{
  author={Machida, Natsumi},
  author={Oguiso, Keiji},
  title={On $K3$ surfaces admitting finite non-symplectic group actions},
  journal={J. Math. Sci. Univ. Tokyo},
  volume={5},
  date={1998},
  number={2},
  pages={273--297},
  issn={1340-5705},
}

\bib{Matsumoto:goodreductionK3}{article}{
  author={Matsumoto, Yuya},
  title={Good reduction criterion for K3 surfaces},
  journal={Math. Z.},
  volume={279},
  date={2015},
  number={1--2},
  pages={241--266},
  issn={0025-5874},
}

\bib{Matsumoto:extendability}{article}{
  author={Matsumoto, Yuya},
  title={Extendability of automorphisms of K3 surfaces},
  year={2021},
  eprint={https://arxiv.org/abs/1611.02092v2},
  journal={Math. Res. Lett.},
  status={to appear},
}

\bib{Nakayama:degeneration}{article}{
  author={Nakayama, Chikara},
  title={Degeneration of $l$-adic weight spectral sequences},
  journal={Amer. J. Math.},
  volume={122},
  date={2000},
  number={4},
  pages={721--733},
  issn={0002-9327},
  label={Naka00},
}

\bib{Nakkajima:logk3}{article}{
  author={Nakkajima, Yukiyoshi},
  title={Liftings of simple normal crossing log $K3$ and log Enriques surfaces in mixed characteristics},
  journal={J. Algebraic Geom.},
  volume={9},
  date={2000},
  number={2},
  pages={355--393},
  issn={1056-3911},
  label={Nakk00},
}

\bib{Nikulin:factorgroups}{article}{
  author={Nikulin, V. V.},
  title={Factor groups of groups of automorphisms of hyperbolic forms with respect to subgroups generated by $2$-reflections. Algebrogeometric applications},
  language={Russian},
  conference={ title={Current problems in mathematics, Vol. 18}, },
  book={ publisher={Akad. Nauk SSSR, Vsesoyuz. Inst. Nauchn. i Tekhn. Informatsii, Moscow}, },
  date={1981},
  pages={3--114},
  note={English translation: J. Soviet Math. {\bf 22} (1983), no. 4, 1401--1475.},
}

\bib{Nygaard:higherdeRham-Witt}{article}{
  author={Nygaard, Niels O.},
  title={Higher de Rham-Witt complexes of supersingular $K3$\^^Msurfaces},
  journal={Compositio Math.},
  volume={42},
  date={1980/81},
  number={2},
  pages={245--271},
  issn={0010-437X},
}

\bib{Ogus:K3crystals}{article}{
  author={Ogus, Arthur},
  title={Supersingular $K3$ crystals},
  conference={ title={Journ\'ees de G\'eom\'etrie Alg\'ebrique de Rennes}, address={Rennes}, date={1978}, },
  book={ series={Ast\'erisque}, volume={64}, publisher={Soc. Math. France, Paris}, },
  date={1979},
  pages={3--86},
}

\bib{RamonMari:Hodgeconjecture}{article}{
  author={Ram{\'o}n Mar{\'{\i }}, Jos{\'e} J.},
  title={On the Hodge conjecture for products of certain surfaces},
  journal={Collect. Math.},
  volume={59},
  date={2008},
  number={1},
  pages={1--26},
  issn={0010-0757},
}

\bib{Rapoport--Zink:monodromie}{article}{
  author={Rapoport, Michael},
  author={Zink, Thomas},
  title={\"Uber die lokale Zetafunktion von Shimuravariet\"aten. Monodromiefiltration und verschwindende Zyklen in ungleicher Charakteristik},
  language={German},
  journal={Invent. Math.},
  volume={68},
  date={1982},
  number={1},
  pages={21--101},
  issn={0020-9910},
}

\bib{Rudakov--Zink--Shafarevich}{article}{
  author={Rudakov, A. N.},
  author={Zink, T.},
  author={Shafarevich, I. R.},
  title={The influence of height on degenerations of algebraic surfaces of type $K3$},
  language={Russian},
  journal={Izv. Akad. Nauk SSSR Ser. Mat.},
  volume={46},
  date={1982},
  number={1},
  pages={117--134, 192},
  issn={0373-2436},
  note={English translation: Math. USSR-Izv. {\bf 20} (1982), no. 1, 119--135 (1983).},
}

\bib{Saito:weightSS}{article}{
  author={Saito, Takeshi},
  title={Weight spectral sequences and independence of $l$},
  journal={J. Inst. Math. Jussieu},
  volume={2},
  date={2003},
  number={4},
  pages={583--634},
  issn={1474-7480},
}

\bib{Schutt:order11}{article}{
  author={Sch\"utt, Matthias},
  title={K3 surfaces with an automorphism of order 11},
  journal={Tohoku Math. J. (2)},
  volume={65},
  date={2013},
  number={4},
  pages={515--522},
  issn={0040-8735},
}

\bib{Schutt:dynamicsssK3}{article}{
  author={Sch\"utt, Matthias},
  title={Dynamics on supersingular K3 surfaces},
  journal={Comment. Math. Helv.},
  volume={91},
  date={2016},
  number={4},
  pages={705--719},
  issn={0010-2571},
}

\bib{Ueno:classification}{book}{
  author={Ueno, Kenji},
  title={Classification theory of algebraic varieties and compact complex spaces},
  series={Lecture Notes in Mathematics, Vol. 439},
  note={Notes written in collaboration with P. Cherenack},
  publisher={Springer-Verlag, Berlin-New York},
  date={1975},
  pages={xix+278},
}

\bib{Zarhin:HodgegroupsK3}{article}{
  author={Zarhin, Yu. G.},
  title={Hodge groups of $K3$ surfaces},
  journal={J. Reine Angew. Math.},
  volume={341},
  date={1983},
  pages={193--220},
}


%\bibselect{myrefs}
\end{biblist}
\end{bibdiv}
\end{document}